\documentclass[11pt]{amsart}
\usepackage{bm}
\usepackage{amssymb}
\usepackage{graphicx}	
\usepackage{xcolor}
\usepackage{amsmath}
\usepackage{enumerate}
\usepackage{blindtext}
\usepackage{multirow}
\usepackage{mathrsfs}

\newtheorem*{notn}{Notations}
\newtheorem{thm}{Theorem}[section]
\newtheorem{Def}[thm]{Definition}
\newtheorem{Rem}[thm]{Remark}
\newtheorem{Cor}[thm]{Corollary}

\newtheorem{Pro}[thm]{Proposition}

\newtheorem{ex}[thm]{Example}
\newtheorem{prob}[thm]{Problem}
\newtheorem{lemma}[thm]{Lemma}

\newcommand{\bdfn}{\begin{Def} \rm}
\newcommand{\edfn}{\end{Def}}

\newcommand{\tfae}{the following are equivalent}

\newcommand{\ra}{\rightarrow}
\newcommand{\Ra}{\Rightarrow}

\newcommand{\ci}{\subseteq}

\newcommand{\al}{\alpha}

\newcommand{\de}{\delta}
\newcommand{\e}{\varepsilon}

\newcommand{\mb}{\mathbb}
\newcommand{\mc}{\mathcal}

\newcommand{\sm}{\setminus}

\newcommand{\iy}{\infty}

\newcommand{\beqa}{\begin{eqnarray*}}
\newcommand{\eeqa}{\end{eqnarray*}}
\newcommand{\vertiii}[1]{{\left\vert\kern-0.25ex\left\vert\kern-0.25ex\left\vert #1 
    \right\vert\kern-0.25ex\right\vert\kern-0.25ex\right\vert}}
\newcounter{cnt1}
\newcounter{cnt2}
\newcounter{cnt3}
\newcounter{cnt4}
\newcommand{\blr}{\begin{list}{$($\roman{cnt1}$)$} {\usecounter{cnt1}
\setlength{\topsep}{0pt} \setlength{\itemsep}{0pt}}}
\newcommand{\blR}{\begin{list}{\Roman{cnt4}.\ } {\usecounter{cnt4}
\setlength{\topsep}{0pt} \setlength{\itemsep}{0pt}}}
\newcommand{\bla}{\begin{list}{$(\alph{cnt2})$} {\usecounter{cnt2}
\setlength{\topsep}{0pt} \setlength{\itemsep}{0pt}}}
\newcommand{\bln}{\begin{list}{$($\arabic{cnt3}$)$} {\usecounter{cnt3}
\setlength{\topsep}{0pt} \setlength{\itemsep}{0pt}}}
\newcommand{\el}{\end{list}}

\begin{document}
\title[Uniqueness of Hahn-Banach extensions]{UNIQUENESS OF HAHN-BANACH EXTENSIONS IN LOCALLY CONVEX SPACES}
\author{Sainik Karak
  \and
  Akshay Kumar
  \and
  Tanmoy Paul
}
\newcommand{\Addresses}{{
  \bigskip
  \footnotesize

  Sainik Karak, \textsc{Indian Institute of Technology Hyderabad,
    India}\par\nopagebreak
  \textit{E-mail address}, Sainik Karak: \texttt{ma22resch11001@iith.ac.in}

  \medskip

  Akshay Kumar, \textsc{Vellore Institute of Technology, Amaravathi Andhra Pradesh, India}\par\nopagebreak
   \textit{E-mail address}, Akshay Kumar: \texttt{akshay.kumar@vitap.ac.in}

  \medskip

  Tanmoy Paul, \textsc{Indian Institute of Technology Hyderabad,
    India}\par\nopagebreak
  \textit{E-mail address}, Tanmoy Paul: \texttt{tanmoy@math.iith.ac.in}
}}

\subjclass[2000]{Primary 46A22, 46A03, 46E10 Secondary 46A99, 46N99 \hfill \textbf{\today} }

\begin{abstract}
We intend to study the uniqueness of the Hahn-Banach extensions of linear functionals on a subspace in locally convex spaces. Various characterizations are derived when a subspace $Y$ has an analogous version of property-U (introduced by Phelps) in a locally convex space, referred to as the property-SNP. We characterize spaces where every subspace has this property.
It is demonstrated that a subspace $M$ of a Banach space $E$ has property-U if and only if the subspace $M$ of the locally convex space $E$ endowed with the weak topology has the property-SNP, mentioned above. This investigation circles around exploring the potential connections between the family of seminorms and the unique extension of functionals previously mentioned. We extensively studied this property on the spaces of continuous functions on Tychonoff spaces endowed with the topology of pointwise convergence.
\end{abstract}

\maketitle

\section{Introduction and Preliminaries}	

\subsection{Introduction}
The uniqueness of the norm-preserving extensions of a continuous linear functional on a subspace of a Banach space is a long-standing and well-studied topic in the literature of Banach spaces. A subspace $M$ of a Banach space $E$ is said to have the {\it property-U} in $E$ if each continuous linear functional on $M$ has a unique norm preserving extension to $E$. Before the seminal study of this property by Phelps in \cite{RP}, Taylor (in \cite{AT}) and Foguel (in \cite{SF}) characterized Banach spaces in which all subspaces possess the property-U. Numerous mathematicians have made a significant contribution to this topic in the last few decades. We refer the reader to \cite{AL,EO,SS} for further details on this topic.

The Hahn-Banach extension theorem for locally convex spaces (in short LCS) is well known (see \cite[Theorem~3.2]{WR}). From the proof of the aforementioned theorem, it clearly follows that not all dominant sublinear functionals ensure a unique Hahn-Banach extension of a given $f$ from the dual of the subspace. A point of departure from this observation led us to study the cases where a subspace $Y$ of an LCS $X$ exhibits a unique Hahn-Banach extension for all continuous linear functionals in $Y^*$.
This investigation answers the following problems. For any unexplained notation, we refer to the notations in Section~1.2.

\begin{prob}\label{P1} 
Let $Y$ be a subspace of an LCS $(X, \tau)$, where $\tau$ is generated by a directed family $\mc{P}$ of seminorms. Let $f \in Y^*$ and $|f| \leq \mu$ on $Y$ for some $\mu \in \mathcal{P}$. Is it always possible to find an extension $\widetilde{f}\in X^*$ of $f$ such that $|\widetilde{f}|\leq\mu$ and $\chi_{\mu'}^Y(f)=\chi_{\mu'}^X(\widetilde{f})$ whenever $\mu'\succeq\mu$? What assumption is necessary for the uniqueness of the extension $\widetilde{f}$? Is it possible to maintain the uniqueness of $\widetilde{f}$ regardless of the selections of $\mu'$?
\end{prob}

A careful review of the proof of \cite[Theorem~3.2]{WR} shows that an $f\in Y^*$ has a unique extension $\widetilde{f}$ on $X$ satisfying $\widetilde{f}\leq \mu$ if for all $x\in X\setminus Y$,
\begin{equation}\label{E1}
\sup_{v\in Y}[f(v)-\mu(v+x)] = \inf_{w\in Y}[-f(w)+\mu(w-x)],
\end{equation}
and vice versa. Here $X, Y, \mu, f$ are as stated in Problem~\ref{P1}.

In \cite{RP} Phelps characterized subspaces with property-U in terms of the Haar property of its annihilators in the dual space. Finite dimensional subspaces in a Banach space are also characterized which have property-U.
While studying property-U of a subspace $M$ of a Banach space $E$, Oja and Lima focused on $M^\#$ (see \cite{AL,EO}). In some cases $M^\#$ plays the role of $M^*$. Our observations in section~2 and 3 extends these notions for locally convex spaces. Let $E$ and $M$ be as above, in contrast to the uniqueness phenomenon of extension of $f$, two separate extensions of $f$ over $E$ with the norm $\|f\|+\e$ exist if $\e>0$ and $f\in M^*$, is observed. In fact, Proposition~\ref{P4} ensures a generalized version of this observation in locally convex spaces.  In section~3, we study the property that is analogous to property-U (referred to as {\it property-SNP}) and its uniform version (referred to as {\it property-USNP}) in quotient spaces. We introduce $Y^\#_{lcs}$ for a subspace $Y$ in a locally convex space $X$. We identify different characteristics of the set $Y_{lcs}^\#$ corresponding to a subspace $Y$ in an LCS $X$ compared to $M^\#$ in a Banach space $E$. Several examples are provided at the end that illustrate the behavior of these properties in locally convex spaces.

\subsection{Preliminaries}

We denote an LCS and the topology endowed on it by the tuple $(X, \tau)$. Suppose $\tau$ is induced by a collection $\mc{P}=\{\rho_i:i\in I\}$ of seminorms on $X$. Then a neighborhood base $\mc{N}$ at $0$ for $(X, \tau)$ is the class of all finite intersections of the sets of the form $\{x\in X: \rho_i(x)<\de\}$, where $i\in I$ and $\de>0$.

The dual of $X$ is denoted by $X^*$ and the annihilator of a subspace $Y$ of $X$ in $X^*$ is denoted by $Y^\perp$. By a hyperplane in $X$, we mean a subspace of the form $f^{-1}(0)$ for some $f\in X^*$. For every $\rho\in\mc{P}$, we define $B_X^\rho:=\rho^{-1}([0,1])$. Let $(X,\tau)$ be as stated above, for a finite subset $F\ci \mc{P}$ and a subspace $Y$ of $X$, we define the following.

\begin{notn}
\bla
\item $B_{X,F}^\rho=\{x\in X: x\in B_X^\rho \text{ for all }\rho\in F\}$
\item $\chi_F^X(f)= \sup\{ |f(x)|: x\in B_{X,F}^\rho\}  \text{ for } f\in X^*$. 
\item $\chi_F^Y(h)= \sup\{ |h(y)|: y\in(B_{X,F}^\rho\cap Y) \} \text{ for all } h\in Y^*$.
\el
\end{notn}

If $F=\{\rho\}$, then $\chi_F^X$ is simply denoted by $\chi_\rho^X$. By an extended seminorm on $X$, we mean a function which satisfies all the properties of a seminorm but can attain infinity (see Example \ref{ex1}).
Since $\rho^{-1}([0, 1])$ is absolutely convex for all $\rho\in F$, $\{\chi_F^X:F\ci \mc{P}\}$ is a family of extended seminorms on $X^*$. We refer the reader to \cite{nwiv, esaetvs} for further details on extended seminorms (norms).  

\begin{ex}\label{ex1} 
Suppose $\tau$ on $c_{00}$ is induced by the seminorms $$\rho_k((x_n)):=\sum_{1\leq j\leq k}|x_j|, \text{ for } (x_n)\in c_{00}.$$ Consider the linear map $f$ defined by $f((x_n))=x_2$ for all $(x_n)\in c_{00}$. Since $|f|= \rho_2$, $f\in c_{00}^*$. It is easy to see that $\sup\{|f(x)|: x\in \rho_1^{-1}([0, 1])\}=\iy$. 
\end{ex}

Although, we have the following in an LCS.

\begin{lemma}\label{lem1} Suppose $(X, \tau)$ be an LCS such that $\tau$ is induced by a family $\mc{P}$ of seminorms. Then $f\in X^*$ if and only if there exists a finite set $F \ci \mc{P}$ (depends on $f$) such that $\chi_F^X(f)<\iy$.
\end{lemma}

Note that if $\tau$ on $X$ is induced by a collection $\mc{P}$ of seminorms, then the family $\mc{P}'$ also induces $\tau$, where $\mc{P}'=\{\sum_{\rho\in F} \rho: F\ci \mc{P} \text{ is finite }\}$. It is easy to see that $\mc{P}'$ is a directed family, that is, for every $\rho_1, \rho_2\in\mc{P}'$, there exists a $\rho_3\in\mc{P}'$ such that $\max\{\rho_1, \rho_2\}\preceq \rho_3$. Therefore, the collection $$\mc{N}=\{\rho^{-1}([0, r)): r>0 \text{ and } \rho\in \mc{P}'\}$$ forms a neighborhood base at $0$ in $(X, \tau)$. Moreover, if $F\ci \mc{P}$, then $\chi_F^X(f)=\chi_\mu^X (f)$ for all $f\in X^*$, where $\mu=\sum\limits_{\rho\in F}\rho\in\mc{P}'$. Therefore, without loss of generality, we can assume that $\mc{P}$ is a directed family. We recall the problem~\ref{P1} stated above that can be partially answered by means of the following.

Let $(X, \tau)$ be an LCS such that $\tau$ is induced by a directed family $\mc{P}$ of seminorms.
Recall from \cite{MR2007455} that for a subspace $Y$ of $X$ and $\rho\in\mc{P}$, an element $y_0\in Y$ is said to be a $\rho$-\textit{best approximation} of an $x_0\in X$ if $$\rho(x_0-y_0)=\inf\{\rho(x_0-y): y\in Y\}.$$

Let us define $d_\rho(x_0,Y)=\inf\{\rho(x_0-y): y\in Y\}$, and also $P_{Y,\rho}(x_0)=\{y\in Y:\rho (x_0-y)=d_\rho(x_0,Y)\}$.

We say $y_0\in Y$ is a {\it unique $\mc{P}$-simultaneous best approximation} of an $x_0\in X$ if 
\[
P_{Y,\rho}(x_0)=\{y_0\} \text{ for all } \rho\in\mc{P}.
\]
In \cite{MR2007455} the authors have considered a weaker notion which was defined as {\it simultaneous best approximation}.
We call an affine subspace $L$ of a vector space $X$ is a {\it line in $X$} if there exists $x_0\in X\sm\{0\}$ such that $L=\{\alpha x_0:\al\in\mb{R}\}+z_0$, for some $z_0\in X$.

\begin{notn}
Let $\mc{P}$ be a given family of seminorms in an LCS $X$.
\bla
\item For $\mu\in\mc{P}$, define $\mc{P}_\mu=\{\rho\in\mc{P}: \rho\succeq\mu\}$.
\item For $f\in X^*$, define $\mc{P}_f=\{\rho\in\mc{P}: \chi_\rho^X (f)<\iy\}$.
\el
\end{notn}

\bdfn\label{D2}
Let $Y$ be a subspace of an LCS $(X,\tau)$ where $\tau$ is endowed by a family $\mc{P}$ of seminorms. Then $Y$ is said to have {\it Haar property} if for $x\in X$ there exists $\mu\in\mc{P}$ such that $x$ has a unique $\mc{P}_\mu$-simultaneous best approximation in $Y$. 
\edfn

The Proposition~\ref{P3} provides a rationale for considering the Definition~\ref{D2} as an analogue of strict convexity. Refer to the Problem~ \ref{P1}, the subsequent lemma facilitates a starting point to answer the first question posed.

\begin{lemma}\label{l2} 
Let $(X, \tau)$ be an LCS such that $\tau$ is induced by a directed family $\mc{P}$ of seminorms. Suppose $Y$ is a subspace of $X$. Then for every $f\in Y^*$ and $\rho\in \mc{P}$ such that $\chi_\rho^Y(f)<\iy$, we have an extension $\widetilde{f}\in X^*$ of $f$ such that $\chi_\rho^Y(f)=\chi_\rho^X(\widetilde{f})$.  
\end{lemma}

The proof of the Lemma~\ref{l2} follows from the Hahn-Banach theorem for locally convex spaces by considering the seminorm $\mu=c\rho$ where $c=\chi^Y_\rho(f)$.

We consider the real scalar field as the underlying scalar field for the spaces. When $E$ is a Banach space, then $B_E$ and $S_E$ represent the closed unit ball and the closed unit sphere of $E$, respectively. By $(E,w)$, we mean the LCS endowed with the weak topology on $E$. We use standard techniques in order to study the aforementioned extension-related problems in locally convex spaces. For a Banach space $E$, we consider the family $\mc{F}=\{\rho_F: F\ci B_{E^*}\mbox{ is finite}\}$, where $\rho_F(x)=\max\{|f(x)|: f\in F\} \text{ for all } x\in E$ in order to define the weak topology on $E$, and this choice leads us to derive the conclusion in Theorem~\ref{SNP and U}, although the family $\{\rho_F:F\ci E^* \mbox{~ is finite}\}$ also induces the weak topology on $E$. We refer the reader to Remark~\ref{R2} in this context.

\section{Property-U in locally convex spaces}

By a pair $(f,\rho)$ on a subspace $Y$ of an LCS $(X, \tau)$ where $\tau$ is induced by a family  $\mc{P}$ of seminorms, we mean $f\in Y^*$, $\rho\in\mc{P}$ such that $\chi_\rho^Y(f)<\iy$. 

\bdfn\label{D1}
Let $(X,\tau)$ be an LCS and $\tau$ be induced by a directed family $\mc{P}$ of seminorms. Let $\rho\in\mc{P}$ and $f\in Y^*$. An $\widetilde{f}\in X^*$ is said to be a Hahn-Banach extension (in short HBE) of $(f, \rho)$ on $Y$ if $\widetilde{f}|_Y=f$ and $\chi_{\rho}^Y(f)=\chi_{\rho}^X(\widetilde{f})$. 
\edfn

In general, a pair $(f, \rho)$ may not have a unique HBE in $X^*$, that is, there may exist two distinct HBEs $\widetilde{f}_1, \widetilde{f}_2\in X^*$ of $f$ such that $\chi_{\rho}^Y(f)=\chi_{\rho}^X(\widetilde{f}_1)=\chi_{\rho}^X(\widetilde{f}_2)$. This leads us to define the following.

\bdfn\label{D3}
Let $(X, \tau)$ be an LCS such that $\tau$ is induced by a directed family $\mc{P}$ of seminorms. 
\bla
\item We call a subspace $Y$ of $(X, \tau)$  has {\it seminorm preserving property} (SNP) if it satisfies anyone of the following equivalent conditions.\blr
\item For $f\in Y^*$, there exists a $\mu\in\mc{P}$ with $\chi_\mu^Y(f)<\iy$ such that all the pairs $(f, \rho)$ for $\rho\in\mc{P}_\mu$ have a unique HBE $\widetilde{f}\in X^*$. 
\item For $f\in Y^*$, there exists a $\mu\in\mc{P}$ such that for every $\rho_1, \rho_2\in\mc{P}_\mu$, both the pairs $(f, \rho_1)$ and $(f, \rho_2)$ have a unique HBE $\widetilde{f}\in X^*$.
 \el

\item We call $Y$ has {\it uniformly seminorm preserving property} (USNP) if it satisfies anyone of the following equivalent conditions.\blr

\item For $f\in Y^*$ and $\rho\in\mc{P}_f$, all the pair $(f, \rho)$ have a unique HBE $\widetilde{f}\in X^*$. 
\item For $f\in Y^*$ and $\rho_1, \rho_2\in \mc{P}_f$, both the pairs $(f,\rho_1)$ and $(f, \rho_2)$ have a unique HBE $\widetilde{f}\in X^*$. \el
\el
\edfn

\begin{Rem}\label{R1} 
If a subspace $Y$ of an LCS $(X, \tau)$ does not have the property-SNP, then there exists an $f\in Y^*$ such that for every $\mu\in \mc{P}$, we have $\rho_1, \rho_2\in \mc{P}_\mu$ such that the pairs $(f, \rho_1)$ and $(f, \rho_2)$ have distinct HBEs $\widetilde{f}_1$ and $\widetilde{f}_2$ in $X^*$, respectively. Similarly, if $Y$ does not have the property-USNP, then there exists an $f\in Y^*$ and $\rho_1, \rho_2\in \mc{P}_f$ such that the pairs $(f, \rho_1)$ and $(f, \rho_2)$ have distinct HBEs $\widetilde{f}_1$  and $\widetilde{f}_2$ in $X^*$, respectively.
\end{Rem}

\begin{Rem}
It should be noted that both the properties SNP and USNP depend on the generating family $\mc{P}$. 
\end{Rem}
 
It is now easy to observe that if $Y$ has USNP in $X$ then the identity in (\ref{E1}) holds for all $\mu\in\mc{P}$ and $x\in X\sm Y$. One can easily see that the condition in (\ref{E1}) for all $\mu\in\mc{P}$ and $x\in X\sm Y$ is not sufficient for USNP for a subspace $Y$. Although we have the following.

\begin{Pro} 
Suppose $Y$ is a subspace of an LCS $(X, \tau)$ such that $\tau$ is induced by a directed family $\mathcal{P}$ of seminorms. Then \tfae. 
\bla
\item $Y$ has the property USNP;
\item for every $f\in Y^*$, $x\in X\setminus Y$, and $\rho, \mu\in \mathcal{P}(f)$, we have
$$\sup_{v\in Y}[f(v)-\chi_{\rho}^Y(f)\rho(v+x)] = \inf_{w\in Y}[-f(w)+\chi_{\mu}^Y(f)\mu(w-x)].$$
\el
\end{Pro}

\begin{Pro} 
Suppose $Y$ is a subspace of an LCS $(X, \tau)$ such that $\tau$ is induced by a directed family $\mathcal{P}$ of seminorms. Then \tfae.
\bla
\item $Y$ has the property SNP;
\item for every $f\in Y^*$, there exists a $\nu\in \mathcal{P}(f)$ such that for all $\rho, \mu\in\mc{P}_{\nu}$ and $x\in X\setminus Y$, we have 
$$\sup_{v\in Y}[f(v)-\chi_{\rho}^Y(f)\rho(v+x)] = \inf_{w\in Y}[-f(w)+\chi_{\mu}^Y(f)\mu(w-x)].$$
\el
\end{Pro} 

In a Banach space $E$ and a subspace $M$ of $E$, $M^\perp$ is a proximinal subspace of $E^*$. One can conclude a similar fact for LCS in the form of the following.

\begin{lemma}\label{lem3}
    Let $Y$ be a subspace of an LCS $(X,\tau)$. Let $h\in X^*$ and $\rho\in\mc{P}$ be such that $d_{\chi_\rho}(h,Y^\perp)<\iy$. Then there exists $g\in Y^\perp$ such that $$\chi_\rho^X(h-g)=d_{\chi_\rho}(h,Y^\perp).$$
\end{lemma}
\begin{proof}
Let $\chi_\rho^X(h-g_\al)\ra d_{\chi_\rho}(h,Y^\perp)$ for some $(g_\al)\ci Y^\perp$. 
From the Banach-Alaglu theorem (see \cite[Theorem~3.15]{WR}), the net $(g_\al)$ has a $w^*$ accumulation point. Without loss of generality, assume that $g_{\alpha}\ra g$ in $w^*$-topology in $X^*$. We see that $\chi_\rho^X(h-g)\leq \lim\sup \chi_\rho(h-g_{\alpha})=d_{\chi_\rho}(h,Y^\perp)$. 
Hence, the result follows.   
\end{proof}

\begin{Pro}\label{P2} 
Let $Y$ be a subspace of an LCS $(X, \tau)$ and let $\tau$ be induced by a directed family $\mc{P}$ of seminorms. If $f\in X^*$, then there exists a $\rho\in\mc{P}$ such that $d_{\chi_{\rho}}(f,Y^{\perp})=\chi_\rho^Y(f|_Y)$, where $d_{\chi_\rho}(f,Y^{\perp})=\inf\{\chi_\rho(f-g):g\in Y^{\perp}\}$.
\end{Pro}

\begin{proof} 
First of all there exists a $\rho\in\mc{P}$ such that $\chi_{\rho}^X(f)<\iy$ and hence $\chi_{\rho}^Y(f|_Y)<\iy$. By the Hahn-Banach theorem, there exists an $h\in X^*$ such that $h|_Y=f|_Y$ and $\chi_{\rho}^X(h)=\chi_{\rho}^Y(f|_Y)$. Then $f-h\in Y^{\perp}$ and $\chi_\rho^Y(f|_Y)=\chi_\rho^X(f-(f-h))\geq d_{{\chi_{\rho}}}(f,Y^{\perp})$. On the other hand if $g\in Y^{\perp}$, then 
  \beqa
      \chi_\rho^Y(f|_Y) &=& \sup\{|(f-g)(y)|:\rho(y)\leq 1\mbox{ and }y\in Y\}\\
                          &\leq & \sup\{|(f-g)(x)|:\rho(x)\leq 1\mbox{ and }x\in X\}\\
                          &=& \chi_\rho^X(f-g).
  \eeqa
This concludes $\chi_\rho^Y(f|_Y)\leq\inf_{g\in Y^\perp}\chi_\rho^X(f-g)$ and hence we get $\chi_\rho^Y(f|_Y)\leq d_{\chi_{\rho}}(f,Y^{\perp})$. Hence, the result follows.   
\end{proof}

Recall the following from \cite{RP}.

\begin{thm}\label{thRP}
    Let $M$ be a subspace of a Banach space $E$. Then $M$ has property-U in $E$ if and only if $M^\perp$ has the Haar property in $E^*$.
\end{thm}

We now derive a similar characterization for the property-(SNP/USNP) in LCS. Let us recall the notion of $\mc{P}$-simultaneous best approximation in a locally convex space, defined in section~1.2.

\begin{thm}\label{th1} 
Let $(X, \tau)$ be an LCS such that $\tau$ is induced by a directed family $\mc{P}$ of seminorms.  
\bla
\item A subspace $Y$ of $(X, \tau)$ has  the property-SNP if and only if $Y^\perp$ has the Haar property in $X^*$. 

\item A subspace $Y$ of $(X, \tau)$ has the property-USNP if and only if for $f\in X^*$, $f$ has a unique $\mc{P}_f$-simultaneous best approximation $\widetilde{f}$ in $Y^\perp$.
\el  
	
\end{thm}
\begin{proof} 
$(a)$. Suppose that $Y$ does not have the property-SNP. Then there exists an $f\in Y^*$ and for every $\mu\in\mc{P}$, there exist $\rho_1,\rho_2\in\mc{P}_\mu$ such that the pairs $(f,\rho_1)$ and $(f,\rho_2)$ have two distinct HBEs $g$ and $h$, respectively. It follows from Proposition~\ref{P2} that,
\[
\chi_{\rho_1}^Y(f)={\chi_{\rho_1}^X}(g)\geq \chi_{\rho_1}^Y(g|_Y)=\inf\{\chi_{\rho_1}^X(g-p): p\in Y^\perp\}=\chi_{\rho_1}^Y(f). 
\]
Thus, $\chi_{\rho_1}^X(g)=\inf\{\chi_{\rho_1}^X(g-p): p\in Y^\perp\}$. Similarly, we can show that $\chi_{\rho_2}^X(g-(g-h))=\chi_{\rho_2}^X(h)=\inf\{\chi_{\rho_2}^X(h-p): p\in Y^\perp\}$. This shows that $P_{Y^\perp,\chi_{\rho_1}^X}(g)=\{0\}$ and $P_{Y^\perp,\chi_{\rho_2}^X}(g)=\{g-h\}$. Since $\rho_1, \rho_2\in\mc{P}_\mu$, $f$ cannot have a unique $\mc{P}_\mu$-simultaneous best approximation in $Y^\perp$. 
	
On the other hand, let $f\in X^*$ and for every $\mu\in \mc{P}$, there exist $\rho_1, \rho_2\in\mc{P}_\mu$ and two distinct $f_1,f_2\in Y^\perp$ such that $P_{Y^\perp,\chi_{\rho_1}^X}(f)=\{f_1\}$ and $P_{Y^\perp,\chi_{\rho_2}^X}(f)=\{f_2\}$. We now consider the pair $(f|_Y,\rho_1)$ on the subspace $Y$. Suppose $\hat{f}\in X^*$ is an HBE of $f|_Y$ such that $\chi_{\rho_1}^X(\hat{f})=\chi_{\rho_1}^Y(f|_Y)$. Then from Proposition~\ref{P2},
\[
\chi_{\rho_1}^X(\hat{f})=\chi_{\rho_1}^Y(f|_Y)=\inf\{\chi_{\rho_1}^X(f-h): h\in Y^\perp\}= \chi_{\rho_1}^X(f-f_1). 
\]
Similarly, we can prove that $\chi_{\rho_2}^Y(f|_Y)=\chi_{\rho_2}^X(f-f_2)$. This shows that pairs $(f|_Y, \rho_1)$ and $(f|_Y, \rho_2)$ have two distinct HBE viz. $f-f_1$ and $f-f_2$. Since $\mu\in\mc{P}$ is arbitrary, $Y$ cannot have the property-SNP.

$(b)$. It is similar to the proof of $(a)$.
\end{proof}
\begin{Rem}
In Banach spaces, the property stated in Theorem~\ref{thRP} is interchangeable when the space is reflexive. The conditions under which the property in Theorem~\ref{th1} can be changed remain unanswered. 
\end{Rem}

Recall the Haar property defined in Definition~\ref{D2}. 

\begin{Pro}\label{P3}
Let $(X,\tau)$ be an LCS, then \tfae.
\bla
\item Each line in $X^*$ has the Haar property in $X^*$. 
\item Any $w^*$-closed subspace $W$ of $X^*$ has the Haar property.
\el
\end{Pro}

\begin{proof}
Let $(b)$ be not true. Then there exists $W\ci X^*$ and there exists a $f\in X^*$ such that any $\mu\in \mc{P}$ there exists $\rho_1,\rho_2\in\mc{P}_\mu$ and $w_{\rho_1}, w_{\rho_2}\in W$, $w_{\rho_1}\neq w_{\rho_2}$, where,
    \[
       \chi_{\rho_1}(w_{\rho_1}-f)= d_{\chi_{\rho_1}}(f,W) \mbox{~and~}
        \chi_{\rho_2}(w_{\rho_2}-f)= d_{\chi_{\rho_2}}(f,W).
    \]
Therefore, the line $L=w_{\rho_1}+\mb{R}z_0$ does not have the Haar property, here $z_0=w_{\rho_2}-w_{\rho_1}$. This concludes $(a)\Ra (b)$.

    $(b)\Ra (a)$. This is obvious.
\end{proof}

We now derive a characterization for locally convex spaces where every subspace has the property-SNP. Similar characterization for Banach spaces was derived by Taylor and Foguel, discussed in the Introduction.

\begin{thm}\label{th2}
  Let $X$ be an LCS. Then every subspace of $X$ has the property-SNP in $X$ if and only if each line in $X^*$ has the Haar property in $X^*$.
\end{thm}
\begin{proof}
Suppose that each line in $X^*$ has the Haar property in $X^*$. Let $Y$ be a subspace of $X$. From Proposition~\ref{P3} it follows that $Y^\perp$ has the Haar property in $X^*$. Consequently, from Theorem~\ref{th1}, $Y$ has the property-SNP.

Conversely, assume that each subspace of $X$ has the property-SNP. Let $L$ be a line in $X^*$. Hence, there exists a $g\in X^*$ such that $L+g$ is a line that passes through the origin. Suppose that $L+g=\mb{R}f$ for some $f\in X^*\sm\{0\}$. Clearly $\ker (f)$ has the property-SNP in $X$ and consequently, by using Theorem~\ref{th1}, $L+g$ has the Haar property in $X^*$. This completes the proof.
\end{proof}

We now present a general result that identifies a class of LCS and their subspaces with the property-SNP but fails to satisfy the property-USNP. Recall that if $E$ is a Banach space, then the weak topology on $E$ is induced by the family of seminorms 
$\mc{F}=\{\rho_F: F\ci B_{E^*}\mbox{ is finite}\}$, where $\rho_F(x)=\max\{|f(x)|: f\in F\} \text{ for all } x\in E$ as stated in the Introduction.

\begin{thm}\label{SNP and U}  
Let $M$ be a subspace of a Banach space $(E, \|\cdot\|)$. Then $M$ has property-U in $E$ if and only if $M$ as a subspace of $(E, w)$ has the property-SNP with respect to the generating family $\mc{F}$ defined above.
\end{thm}

\begin{proof} 
Suppose $M$ does not have property-U in $E$. Then there exists an $f\in S_{M^*}$ which has two distinct HBEs $f_1, f_2\in S_{E^*}$. Let $F\ci B_{E^*}$ be any finite set with $\chi_{\rho_F}^M(f)<\iy$. 

{\sc Claim:~} There exist distinct HBEs $f_1, f_2$ of the pair $(f,\rho_{F_1})$, for some $F_1\supseteq F$, $F_1\ci B_{E^*}$ finite. 

Consider $F_1=F\cup \{f_1,f_2\}$. It is enough to show that the pair $(f,\rho_{F_1})$ has two distinct HBEs $f_1,f_2$. Clearly, $f_1|_Y=f=f_2|_M$. Since $F_1\ci B_{E^*}$, we have $1=\|f_1\|=\|f\|\leq \chi_{\rho_{F_1}}^M(f)\leq \chi_{\rho_{F_1}}^E(f_1)$. As $f_1\in F_1$, $\chi_{\rho_{F_1}}^E(f_1)\leq 1$. Thus, $\chi_{\rho_{F_1}}^E(f_1)=\chi_{\rho_{F_1}}^M(f)=1$. Similarly, we can prove that $\chi_{\rho_{F_1}}^E(f_2)=\chi_{\rho_{F_1}}^M(f)=1$. This concludes that $M$ does not have the property-SNP.
	
Conversely, suppose $M$ does not have the property-SNP. It follows from Remark~\ref{R1} that there exists a $f\in M^*$ such that for every finite set $F\ci B_{E^*}$, we have finite sets $F_1,F_2\supseteq F$ ($F_1$ may be equal to $F_2$) and two distinct HBEs $f_1, f_2\in E^*$ of the pairs $(f,\rho_{F_1}), (f,\rho_{F_2})$ satisfying 
\[
	\chi_{\rho_{F_1}}^E(f_1)=\chi_{\rho_{F_1}}^M(f) \mbox{~and~}
	\chi_{\rho_{F_2}}^E(f_2)=\chi_{\rho_{F_2}}^M(f).
\]
Without loss of generality, assume that $\|f\|=1$. Now, let $\widetilde{f}\in S_{E^*}$ be a HBE of $f$, in the usual sense. Consider $F=\{\widetilde{f}\}$. From our assumption on $M$, we have $F_1,F_2\supseteq F$ and $f_1,f_2\in E^*$ such that the above holds. It is now easy to prove that $\chi_{\rho_{F_1}}^E(f_1)=\chi_{\rho_{F_1}}^M(f)=\|f\|=\chi_{\rho_{F_2}}^M(f)=\chi_{\rho_{F_2}}^E(f_2)$. Since  $\|f_1\|\leq \chi_{\rho_{F_1}}^E(f_1)$, we have $\|f\|=\|f_1\|$. Similarly, we can prove that $\|f\|=\|f_2\|$. Therefore, $f\in M^*$ has two distinct norm-preserving extensions. Hence, $M$ does not have property-U. \end{proof}

\begin{Rem}\label{R2} 
In Theorem \ref{SNP and U}, we cannot replace $B_{E^*}$ by $E^*$ as property-U for a subspace $M$ always depends on the norm. Take $E=(\mathbb{R}^n,\|.\|_\iy)$ and $Z=(\mathbb{R}^n,\|.\|_1)$ for $n>1$, then one can see that $E^*=Z^*$ as a set. However, they can have different subspaces which have property-U.
\end{Rem}

We now show that we cannot replace the property-SNP with the property-USNP in Theorem \ref{SNP and U}.

\begin{notn}
Let $M$ be a subspace of a normed space $E$ and $f\in M^*$. We define $extn(f)=\{h\in E^*:h|_M=f\}$.
\end{notn}

\begin{Pro}\label{P4}
    Let $Y$ be a subspace of an LCS $(X,\tau)$, where $\tau$ is induced by the family $\mathcal{P}$ of seminorms and $f\in Y^*$. Also let $(f,\rho)$ be a pair and $r>\chi_\rho^Y(f)$. Then there exist distinct $f_1,f_2\in extn(f)$ such that $\chi_\rho^X(f_1)=\chi_\rho^X(f_2)=r$.  
\end{Pro}

\begin{proof}
    Let $g(\neq 0)\in Y^\perp$ and $\rho'\in\mathcal{P}$ such that $0\neq\chi_{\rho'}^X(g)<\infty$. Now choose $\mu\geq max\{\rho,\rho'\}$, then $\chi_\mu^X(g)$ and $\chi_\mu^Y(f)$ are finite. So there exists an HBE $\widetilde{f}$ of the pair $(f,\mu)$. Without loss of generality, assume that $\chi_\mu^X(\widetilde{f})=\chi_\mu^Y(f)=1$. Clearly, $extn(f)\supseteq \{\widetilde{f}+\al g\}$ and define $\{\widetilde{f}+\al g:\al\in\mathbb{R}\}=L$. It remains to prove that for $r>1$ there exist two distinct $f_1,f_2\in L$ such that $\chi_\mu^X(f_1)=\chi_\mu^X(f_2)=r$. 

    Now define $\varphi :[0,\infty)\ra \mathbb{R}$ by $\varphi(\al)=\chi_\mu^X(\widetilde{f}+\al g)$ and $\psi:(-\iy,0]\ra \mb{R}$ by $\psi(\al)=\chi_\mu^X(\widetilde{f}+\al g)$. It is easy to observe that  both $\varphi, \psi$ are continuous and $\varphi(0)=\psi(0)=1$. So by IVP for $r>1$, there exists $\al_0>0$ and $\beta_0<0$ such that $\chi_\mu^X(\widetilde{f}+\al_0 g)=\chi_\mu^X(\widetilde{f}+\beta_0 g)=r$. Clearly $\widetilde{f}+\al_0 g,\widetilde{f}+\beta_0 g\in L$ are distinct. Hence, the result follows.
\end{proof}

As an immediate consequence of Proposition~\ref{P4}, we have the following.

\begin{Cor}\label{cor1}
Let $M$ be a subspace of a normed space $E$ and $f\in M^*$. Then,
\bla
\item there exist $f_1, f_2\in extn(f)$ such that $f_1\neq f_2$  and $\|f_1\|=\|f_2\|$.  
\item for $r>\|f\|$, there exist at least two distinct $\widetilde{f}_1, \widetilde{f}_2\in extn(f)$ such that $\|\widetilde{f}_i\|=r$, $i=1,2$.
\el
\end{Cor}

\begin{thm}
Let $M$ be a subspace of a Banach space $E$. Then $M$ can not have the property-USNP in $E$ with respect to the weak topology on $E$.
\end{thm}
\begin{proof}
Let $f\in M^*$ be such that $\|f\|=1$. Suppose $f$ is norm attaining, that is $f(y)=1$ for some $y\in S_M$. By Corollary~\ref{cor1}, for $r> 1$, there exist $f_1,f_2\in E^*$ with $f_1|_M=f=f_2|_M$ and $\|f_1\|=\|f_2\|=r$. Denote $F=\{\frac{f_1}{r},\frac{f_2}{r}\}$ and consider $B_E^{\rho_F} (=\rho_F^{-1}[0,1])$. Clearly $\chi_{\rho_F}^E(f_1)=r=\chi_{\rho_F}^E(f_2)$ and $ry\in B_E^{\rho_F}$. Clearly, $\chi_{\rho_F}^E(f)\leq r$ and $f(ry)=r$, so $\chi_{\rho_F}^E(f)=r$. Hence, $f_1$ and $f_2$ are two distinct HBEs of the pair $(f,\rho_F)$. Hence $M$ does not have the property-USNP in $E$ with respect to weak topology.
\end{proof}    

\section{$Y_{lcs}^\#$ and property-SNP}

Let $M$ be a subspace of a Banach space $E$. Let us recall the set $M^\#$ introduced by Oja in \cite{EO} in relation to study property-U of $M$, where $M^\#=\{f\in E^*:\|f\|=\|f|_Y\|\}$.
This section defines $Y_{lcs}^\#$ and examines certain vector space properties of it and its connection with the property-SNP.

Suppose $(X, \tau)$ be an LCS and $\tau$ is induced by a directed family $\mc{P}$ of seminorms. For a subspace $Y$ of $(X, \tau)$ and $\mu\in \mc{P}$, we define 
$$Y^{\#}_\mu= \left\lbrace f\in X^*: \chi_\rho^X(f)=\chi_\rho^Y(f|_Y)<\iy \text{ for all } \rho\in\mc{P}_\mu \right\rbrace.$$  

\begin{Pro}\label{Y coicides in normed space} 
Let $M$ be a subspace of a Banach space $E$. Then $M^\#=\bigcup_{\mu\in\mc{F}} M^\#_\mu,$ where $\mc{F}=\{\rho_F:F\ci B_{E^*} \mbox{~is finite}\}$.  
\end{Pro}

\begin{proof} 
Assume that $f\in M^\#_{\rho_F}$ for some finite $F\ci B_{E^*}$. First assume that $\|f|_M\|\neq 0$. Without loss of generality, assume that $\|f|_M\|=1$ (otherwise, consider $\frac{f}{\|f|_M\|}$). Then $\chi_{\rho}^E(f)=\chi_{\rho}^M(f|_M)$ for all $\rho\succeq \rho_F$. Note that for $F_1=F\cup\{f\}$, $\rho_{F_1}\succeq \rho_{F}$ and consequently, $\chi_{\rho_{F_1}}^E(f)=\chi_{\rho_{F_1}}^M(f|_M)$. It is easy to observe that 
$1=\|f|_M\|\leq \|f\|\leq \chi_{\rho_{F_1}}^E(f)\leq 1$. 
Hence, $\|f\|=\|f|_M\|$. Now, if $\|f|_M\|=0$, then $\chi_{\rho_{F_1}}^E(f)=\chi_{\rho_{F_1}}^M(f|_M)=0$. Since $\|f\|\leq \chi_{\rho_{F_1}}^E(f)$, $\|f\|=0$.

For the reverse inclusion, let $f\in M^\#$. Without loss of generality, assume that $\|f|_M\|=1=\|f\|$. Let $F=\{f\}$. Then for every finite set $F_1\supseteq F$, $F\ci B_{E^*}$ we have $1=\|f|_M\|\leq  \chi_{\rho_{F_1}}^M(f|_M)\leq 1$. Similarly, $1=\|f\|\leq  \chi_{\rho_{F_1}}^E(f)\leq 1$. Consequently, $1=  \chi_{\rho_{F_1}}^M(f|_M)=\chi_{\rho_{F_1}}^E(f)$. Hence, $f\in M^\#_F\ci M^\#$. \end{proof}

Proposition~\ref{Y coicides in normed space} motivates to define the following.

\bdfn
Suppose $(X, \tau)$ is an LCS such that $\tau$ is induced by a directed family $\mc{P}$ of seminorms. We define $Y_{lcs}^\#:=\bigcup_{\mu\in\mc{P}} Y^\#_\mu.$
\edfn

\begin{lemma}\label{Y is linear then ext is atmost 1} 
Suppose $Y$ is a subspace of an LCS $(X, \tau)$ where $\tau$ is induced by a directed family $\mc{P}$ of seminorms. Suppose that $Y_{lcs}^\#$ is linear. Then every $f\in Y^*$ and $\mu\in\mc{P}$, there exists at most one $\widetilde{f}\in Y_\mu^\#$ such that $\widetilde{f}|_Y=f$. 
\end{lemma}

\begin{proof} 
If possible, assume that there exist $f_1\in Y_{\mu_1}^\#$ and $f_2\in Y_{\mu_2}^\#$ be two distinct extensions of $f$, for some $\mu_1, \mu_2\in\mc{P}$. Since $Y^\#$ is linear, $f_1-f_2\in Y^\#_{\mu_3}$ for some $\mu_3\in\mc{P}$. Consequently, $\chi_{\mu_3}^X(f_1-f_2)=\chi_{\mu_3}^Y((f_1-f_2)|_Y)=0$ as $\mc{P}$ is a directed family of seminorms. This concludes $f_1-f_2=0$. Hence, $f_1=f_2$. 
\end{proof}

\begin{Rem}  
Example~\ref{ex4} shows that there exists an LCS $X$ and for its subspace $Y$, $Y_{lcs}^\#=\{0\}$.
\end{Rem}

\begin{thm} 
Let $M$ be a subspace of a Banach space $E$ such that $M^\#$ is linear. Then $M$ as a subspace of $(E, w)$ has the property-SNP. 
\end{thm}
 
\begin{proof} It follows from Theorem \ref{SNP and U}, Proposition \ref{Y coicides in normed space} and \cite{EO}\end{proof}

We now extend the result in \cite[Theorem~2.1]{AL} concerning the set $M^\#$ in locally convex spaces.

\begin{thm}\label{th3}
Let $Y$ be a subspace of a LCS $(X,\tau)$ where $\tau$ is induced by a directed family of $\mc{P}$ seminorms . Consider the following.
\bla
\item $Y$ has the property-SNP in $X$.
\item  Every member of $X^*$ can be expressed as a unique sum of members of $Y_{lcs}^\#$ and $Y^\perp$.
\item  If $f_1,f_2\in Y_{lcs}^\#$ such that $f_1+f_2\in Y^\perp$, then $f_1+f_2=0$. 
\el
Then $(a)\Ra (b)\Ra (c)$. 
\end{thm}

\begin{proof}
$(a)\Ra (b).$ Suppose that $Y$ has the property-SNP and $f\in X^*$. There exists a $\mu\in\mc{P}$ such that every pair $(f|_Y,\rho)$ on $Y$ has a unique HBE $\hat{f}\in X^*$ for all $\rho\in\mc{P}_\mu$. This follows $\hat{f}\in Y_\mu^\#$. Consequently, $\hat{f}\in Y_{lcs}^\#$. Clearly, $f=\hat{f}+f-\hat{f}$. It remains to show that this decomposition is unique. If possible, let $f=g+h$ for some $g\in Y_{lcs}^\#$ and $h\in Y^\perp$. Then, $g\in Y_{\mu_1}^\#$ for some $\mu_1\in\mc{P}$. Note that $\chi_\rho^X(g)=\chi_\rho^Y(g|_Y)=\chi_\rho^Y(f|_Y)=\chi_\rho^X(\hat{f})$ for all $\rho\succeq\max \{\mu,\mu_1\}$. This implies that $\hat{f}=g$ and $f-\hat{f}=h$. Hence, the result follows.

$(b)\Ra (c).$ Suppose that $(c)$ is not true. Then there exist $f_1,f_2\in Y_{lcs}^\#$ such that, $f_1+f_2\in Y^\perp$ but $f_1+f_2\neq 0$. As $f_1,f_2\in Y_{lcs}^\#$, it follows that $\pm f_1,\pm f_2\in Y_{lcs}^\#$. Now we can write, 
$-f_2+(f_1+f_2)=f_1=f_1+0.$
Since $f_1\neq -f_2$ and $-f_2\in Y_{lcs}^\#$, we have two decompositions of $f_1$, hence $(b)$ is not true. 
\end{proof}

\begin{Rem}
\bla
\item In Theorem~\ref{th3}, $(c)$ does not imply $(a), (b)$. From Example~\ref{ex4} one can see that $Y_{lcs}^\#=\{0\}$, which follows that $Y$ has property $(c)$ in Theorem~\ref{th3} but does not have property $(a)$. Also, $Y$ may not always have the property $(b)$ otherwise $Y^\perp=X^*$, which may not be true. 
\item If we assume $Y_{lcs}^\#$ to be linear in condition $(b)$, then $(b)\Ra (a)$. This follows from the Lemma~\ref{Y is linear then ext is atmost 1}. 
\el
\end{Rem}

Proposition~\ref{Y coicides in normed space} and Theorem~\ref{th3} motivate us to raise the following problem.

\begin{prob}
Find all locally convex spaces $(X, \tau)$ and its subspaces $Y$ such that one of the statements specified in Theorem~\ref{th3} $(b)$ and $(c)$ can characterize the property-SNP of $Y$.
\end{prob}

We now study these properties in quotient spaces. The following example ensures that the property-SNP may not pass through the quotient spaces.

\begin{ex}
Consider $X=(\mb{R}^3,\rho,\mu), Y=\{(x,y,0):x,y\in\mathbb{R}\}$ and $Z=\{(x,0,0):x\in\mathbb{R}\}$, where $\rho=\|.\|_\iy$ and $\mu=\|.\|_1$. Then $(X/Z,\dot{\rho},\dot{\mu})$ coincide with $(X/Z,\|.\|_\iy,\|.\|_1)$ and $Y/Z$ is the $Y$-axis which does not have the property-SNP in $X/Z$ (see \cite[Theorem~2.3]{RP}). But $Y$ has the property-SNP in $X$.
\end{ex}

We have an affirmative answer for the property-USNP.
Recall that if $Z$ is a closed subspace of an LCS $(X, \tau)$ and $\tau$ is induced by a family $\mc{P}$ of seminorms, then the quotient topology $\pi(\tau)$ on the quotient space $X/Z$ is induced by the family $\dot{\mc{P}}=\{\dot{\rho}: \rho\in \mc{P}\},$ where $\dot{\rho}(x+Z)= \inf\{\rho(x-z): z\in Z\}$. For more details on a quotient space in locally convex spaces, we refer the readers to \cite{tvsnarici}.    

\begin{lemma}\label{lem2}
Suppose that $Z$ is a closed subspace of an LCS $(X, \tau)$ and $\rho\in\mc{P}$ where $\tau$ is induced by a family $\mc{P}$ of seminorms. Then for every $f\in (X/Z)^*$, we have $\chi_{\rho}^X(f\circ\pi_{X/Z})=\chi_{\dot{\rho}}^{X/Z}(f)$.\end{lemma}

\begin{proof} Since $\dot{\rho}(x+Z)\leq \rho(x)$,  $\chi_{\rho}^X(f\circ\pi_{X/Z})\leq\chi_{\dot{\rho}}^{X/Z}(f)$. Now, suppose $x\in X$ such that $\dot{\rho}(x+Z)<1$, then for $\e= \frac{1-\dot{\rho}(x+Z)}{2}$, there exists an $z\in Z$ such that $\rho(x-z)<\dot{\rho}(x+Z)+\e$. Which implies that $\rho(x-z)<1$. Note that $|f(x+Z)|=|f\circ \pi_{X/Z}(x-z)|\leq \chi_{\rho}^X(f\circ\pi_{X/Z})$. Thus, $\chi_{\dot{\rho_1}}^{X/Z}(f)\leq \chi_{\rho}^X(f\circ\pi_{X/Z})$. Hence, $\chi_{\dot{\rho}}^{X/Z}(f)= \chi_{\rho}^X(f\circ\pi_{X/Z})$.   
\end{proof}

\begin{thm}\label{th4}
Let $X, Z, \tau, \mc{P}$ be as above. Let $Y$ be a closed subspace of $X$. Then \tfae.
\bla
\item $Y$ has the property-USNP in $X$.
\item For every closed subspace $Z(\neq 0)$ of $Y,\;Y/Z$ has the property-USNP in $X/Z$.
\item For every nonzero hyperplane $Z$ in $Y$, $Y/Z$ has the property-USNP in $X/Z$.\el
\end{thm}

\begin{proof} $(a)\Rightarrow(b)$. Let $Z$ be a closed subspace of $Y$ such that $Y/Z$ does not have the property-USNP in $X/Z$. So, there exists an $f\in (Y/Z)^*$ for which there exist $\dot{\rho_1},\dot{\rho_2}\in\dot{\mc{P}}$ such that the pairs $(f,\dot{\rho_1})$ and $(f,\dot{\rho_2})$ have two distinct HBEs $g$ and $h$ on $X/Z$, respectively. Clearly, $f\circ \pi_{Y/Z}\in Y^*$ and $g\circ\pi_{X/Z},h\circ\pi_{X/Z}\in X^*$ and $g\circ\pi_{X/Z}\neq h\circ\pi_{X/Z}$. It is now enough to show that $g\circ\pi_{X/Z}$ and $h\circ\pi_{X/Z}$ are two HBEs of  the pairs $(f\circ\pi_{Y/Z},\rho_1)$ and $(f\circ\pi_{Y/Z},\rho_2)$, respectively. By Lemma~\ref{lem2}, we have $\chi_{\rho_1}^Y(f\circ\pi_{Y/Z})=\chi_{\dot{\rho_1}}^{Y/Z}(f)=\chi_{\rho_1}^X(g\circ\pi_{X/Z})=\chi_{\dot{\rho_1}}^{X/Z}(g)$ and  $\chi_{\rho_2}^Y(f\circ\pi_{Y/Z})=\chi_{\dot{\rho_2}}^{Y/Z}(f)=\chi_{\rho_2}^X(h\circ\pi_{X/Z})=\chi_{\dot{\rho_2}}^{X/Z}(h)$. This completes the proof.	
	
The implication $(b)\Rightarrow(c)$ is immediate.  
	
$(c)\Rightarrow(a)$. Assume that $Y$ does not have the property-USNP in $X$. There exists a non-zero $f\in Y^*$ and $\rho_1,\rho_2\in\mc{P}$ such that pairs $(f,\rho_1)$ and $(f,\rho_2)$ have two distinct HBEs $f_1$ and $f_2$, respectively. Consider $Z$=ker$(f)$. Then $Z\ci Y$ is a hyperplane in $Y$. Since $f|_Z=0$, all the maps $f_1\circ \pi_{X/Z}^{-1}, f_2\circ \pi_{X/Z}^{-1}$ and $f\circ\pi_{Y/Z}^{-1}$ are well defined. Note also that $f_1\circ \pi_{X/Z}^{-1}, f_2\circ \pi_{X/Z}^{-1}\in (X/Z)^*$ and $f\circ\pi_{Y/Z}^{-1}\in (Y/Z)^*$. Clearly, $f_1\circ\pi_{X/Z}^{-1}\neq f_2\circ\pi_{X/Z}^{-1}$. It is similar to the Lemma~\ref{lem2} that $\chi_{\dot{\rho}_1}^{Y/Z}(f\circ \pi_{Y/Z}^{-1})=\chi_{\rho_1}^Y(f)=\chi_{\rho_1}^X(f_1)=\chi_{\dot{\rho}_1}^{X/Z}(f_1\circ\pi_{X/Z}^{-1})$ and $\chi_{\dot{\rho}_2}^{Y/Z}(f\circ\pi_{Y/Z}^{-1})=\chi_{\rho_2}^Y(f)=\chi_{\rho_2}^X(f_2)=\chi_{\dot{\rho}_2}^{X/Z}(f_2\circ\pi_{X/Z}^{-1})$.
This shows that $Y/Z$ does not have the property-USNP in $X/Z$. Hence, the result follows.
\end{proof}

\section{Illustrative Examples}

Our first example shows that for a locally convex space $X$, and its subspace $Y$ it is possible that $Y_{lcs}^\#=\{0\}$.
\begin{ex}\label{ex4} 
Consider the LCS $(\mathbb{R}^2, \tau)$, where $\tau$ is induced by the directed family $\mc{P}=\{\rho_n:n\in\mathbb{N}\}$ of the seminorms given by the following.

$$\rho_n((x,y))=\Bigg\{\begin{array}{ll} n(|x|+|y|) & \text{ if $n=2^m$ for some $m$}\\ n(|x|+\frac{|y|}{2}) &\text{ if $n$ is odd}\\ n\max\{|x|, |y|\} & \text{ otherwise}.\end{array}$$ 

Consider $Y=\{(x, y)\in \mathbb{R}^2:x=y\}$ and $f\in Y^*$, where $f(x,x)=kx$ for some non-zero $k\in\mathbb{R}$. Let $\mu\in \mc{P}$ and $\phi\in X^*$ be an HBE of the pair $(f,\mu)$. Let $\mu=\rho_{n_0}$.
Hence we have $\phi|_Y=f$ and $\chi_\rho^Y(f)=\chi_\rho^X(\phi)<\iy$ for all $\rho\in\mc{P}_\mu$. We can assume that $\phi(x, y)=\alpha x+\beta y$ for all $x, y\in\mathbb{R}$. Since $\phi|_Y=f$, we have $\alpha+\beta=k$.

{\sc Step 1:~}Take $n\geq n_0$ such that $n=2^m$ for some $m$, and $\rho_n\in\mc{P}_\mu$. Then $$\frac{|k|}{2n}=\chi_{\rho_n}^Y(f)=\chi_{\rho_n}^X(\phi)=\frac{1}{n}\max\{|\alpha|, |\beta|\}.$$ 
    Consequently, $\alpha=\beta=\frac{k}{2}$. Thus, $\phi(x, y)=\frac{k}{2}(x+y)$ for all $x, y\in \mathbb{R}$.

{\sc Step 2:~}Take $n\geq n_0$ such that $n$ is odd and $\rho_n\in\mc{P}_\mu$. Then $$\frac{|2k|}{3n}=\chi_{\rho_n}^Y(f)=\chi_{\rho_n}^X(\phi)\geq\phi\left(\left(\frac{1}{2n}, \frac{1}{n}\right)\right)=\frac{3|k|}{4n}.$$  

Which is not possible. This shows that $Y^\#_\mu=\{0\}$. Consequently, $Y_{lcs}^\#=\{0\}$. 

{\sc Step3:~} For every even $n\in\mb{N}$ such that $n\neq 2^m$ for any $m$, the pair $(f, \rho_n)$ have two distinct HBEs $\phi_1$ and $\phi_2$ in $X^*$ given by 
\[
\phi_1((x,y))=kx, \mbox{ and }\phi_2((x,y))=\frac{k}{2}(x+y).
\]

Hence, $Y$ does not have the property-SNP though $Y_{lcs}^\#$ is linear. 
\end{ex}

For the remaining of this section we consider $Z$ as a Tychonoff space where the LCS $(C_p(Z),\tau)$ represents the topology of pointwise convergence on the set of all continuous functions on $Z$. It is well-known that $\tau$ is induced by the directed family 
\[
\mc{P}=\left\lbrace \rho_F: F\ci Z \text{ is finite}\right\rbrace
\]
of seminorms, where $\rho_F(f)=\max_{x\in F} |f(x)|$ for all $f\in C_p(Z)$.

\begin{ex}\label{finite max pointwise family} 
Consider the subspace $Y=\{f\in C_p(Z): f|_{A}=0\}$.  Now, suppose $(\phi, \rho_P)$ is any pair on $Y$ for some $\rho_P\in \mc{P}$ and $\phi(f)=\sum_{i=1}^{n} \alpha_if(x_i) $ for all $f\in Y$. Without loss of generality, we can assume that all $x_i\notin [0, 1]$ and no $\alpha_i=0$. Since $\chi_{\rho_P}^Y(\phi)<\iy$, each $x_i\in P$.  Consider $\widetilde{\phi}(f)=\sum_{i=1}^{n} \alpha_if(x_i) $ for all $f\in C_p(Z)$. Note that $\chi_{\rho_P}^Y(\phi)=\chi_{\rho_P}^X(\widetilde{\phi})=\sum_{i=1}^n|\alpha_i|$ and $\widetilde{\phi}|_Y=\phi$. Let $\psi\in (C_p(Z))^*$ is an extension of the pair $(\phi, \rho_P)$. Assume that $\psi=\sum_{j=1}^{m}\beta_jf(y_j)$ for all $f\in C_p(Z)$. Then $y_j\in P$ for all $j$ as $\chi_{\rho_P}^X(\psi)<\iy$. 
	
{\sc Step 1:} We prove that each $x_i\in \{y_j:1\leq j\leq m\}$. If $x_i\notin \{y_j:1\leq j\leq m\}$ for some $i$, then we can find an $f\in C_p(Z)$ such that $f|_{A\cup \{y_j:1\leq j\leq m\}}=0$ and $f(x_i)\neq 0$. Consequently, $0\neq \phi(f) =\psi(f)=0$. Which is not possible.
	
{\sc Step 2:} If for some $k$, $y_k \notin A$ is distinct from all $x_i$ and $\beta_k\neq 0$, then we can find an $f\in C_p(Z)$ such that $f|_{A\cup\{y_j:j\neq k\}\cup \{x_i: 1\leq i\leq n\} }=0$ and $f(y_k)=1$. Which implies that $\beta_k=\psi(f)=\phi(f)=0$. Which is not possible. Therefore, $y_k=x_i$ for some $i$ and $\beta_k=\psi(f)=\phi(f)=\alpha_i$.   
	
{\sc Step 3:} If $y_j\in A$ and $\beta_j\neq 0$ for some $j$, then $\sum_{j=1}^m|\beta_j|> \sum_{i=1}^n|\alpha_i|$ by using Step 1 and Step 2. Which is not possible as $\sum_{j=1}^m|\beta_j|=\chi_{\rho_P}^X(\psi)=\chi_{\rho_P}^Y(\phi)=\sum_{i=1}^n|\alpha_i|$. Consequently, $\psi=\widetilde{\phi}$. Hence, $Y$ has the property-USNP in $C_p(Z)$. 
\end{ex}

We next provide an example of a subspace that does not have the property-SNP.

\begin{ex}\label{finite sum pointwise family} 
Suppose that $A\ci Z$ is not empty and $Y$ be as above. It is well known that the directed family $\mc{P}'=\{\rho_F: F\ci Z \text{ is finite}\}$ induces the topology of the LCS $X=C_p(Z)$, where $\rho_F(f)=\sum_{x\in F}|f(x)|$ for all $f\in C(Z)$. We show that $Y$ does not have the property-SNP. Let $\phi(f)=\sum_{i=1}^{n} \alpha_if(x_i) $ for all $f\in Y$. Without loss of generality, we can assume that all $\alpha_i$ are non-zero and $x_i\notin A$ for all $1\leq i\leq n$. Let $\rho_F\in \mc{P}'$. Clearly, $\chi_{\rho_F}^Y(\phi)<\iy$ if and only if all $x_i\in F$. Note that $\chi_{\rho_F}^Y(\phi)=\max_{1\leq i\leq n}|\alpha_i|=|\alpha_k|$ (say). Consider $P=F\cup\{z\}$ for some $z\in A$. Clearly, $\rho_P\succeq \rho_F$ and $\chi_{\rho_P}^Y(\phi)=|\alpha_k|$.  Consider $$\widetilde{\phi}_1(f)=\sum_{i=1}^{n} \alpha_if(x_i) \text{ for all }f\in C(Z), \text{ and }$$
	$$\widetilde{\phi}_2(f)=\sum_{i=1}^{n} \alpha_if(x_i)+\frac{|\alpha_k|}{2}f(z)  \text{ for all }f\in C(Z).$$ Observe that $\chi_{\rho_P}^X(\widetilde{\phi}_1)=\chi_{\rho_P}^X(\widetilde{\phi}_2)= |\alpha_k|$ and $\widetilde{\phi}_1|_{Y_A}=\widetilde{\phi}_2|_{Y_A}=\phi$. Therefore, the pair $(\phi, \rho_P)$ has two distinct extensions. Hence, $Y_A$ does not have the property SNP.    
\end{ex}

\begin{ex}\label{ex3} 
Suppose that $A\ci Z$ is nonempty. Let us consider the following generating family of seminorms $$\mc{P}=\{\rho_F:F\ci Z \text{ is finite}\}\cup\{\mu_F:F\ci Z \text{ is finite}\}$$ for the LCS $C_p(Z)$, where $\rho_F(f)=|F|\max_{x\in F}|f(x)|,$ and $\mu_F(f)=|F|\sum_{x\in F}|f(x)|$ for all $f\in C(Z)$, here $|F|$ represents the cardinality of the set $F$. Let $Y$ be as above. Suppose $\phi\in (C_p(Z))^*$. Assume that $\phi=\sum_{i=1}^n\alpha_i \delta_{x_i}$ with $x_1, \cdots, x_k\in A$ and $x_{k+1}, \cdots, x_n\in Z\setminus A$. Consider $\phi_1=\sum_{i=k+1}^n\alpha_i \delta_{x_i}$ and $\phi_2=\sum_{i=1}^k\alpha_i \delta_{x_i}$. Note that for $F=\{x_i:1\leq i\leq n\}$, $\phi_1\in Y^\#_{\rho_F}$ and $\phi_2\in Y^\perp$ with $\phi=\phi_1+\phi_2$. It is similar to Example \ref{finite sum pointwise family} that for $z\in A$ and $F_1\supseteq F$, the pair $(\phi|_Y, \mu_{F_1})$ has two distinct extensions $\phi_1$ and 
$$\psi_1(f)=\sum_{i=1}^{n} \alpha_if(x_i)+ rf(z)  \text{ for all }f\in C(Z),$$ where $r=\max\{|\alpha_i|: k+1\leq i\leq n\}$. Thus, $Y$ does not have the property-SNP. It is similar to Example \ref{finite max pointwise family} that if $\psi_2\in (C_p(Z))^*$ is an extension of the pair $(\phi|_Y, \rho_{F_2})$ for some $F_2\supseteq F$, then $\psi_2$ has to be equal to $\phi_1$. Hence, every member of $(C_p(Z))^*$ can be written as a unique sum of members of $Y^\#$ and $Y^\perp$. 
\end{ex}

We now discuss the property-SNP for one-dimensional subspaces of $C_p(Z)$. 

\begin{Pro}\label{prop 1} 
Suppose $X=C_p(Z)$ and $f_0\in X$ such that $f_0$ does not attain its supremum value. Then $Y$= span$(f_0)$  does not have the property-SNP with respect to $\mc{P}=\{\rho_F:F\ci Z \text{ is finite}\}$, where $\rho_F(f)=\max_{x\in F}|f(x)|$ for all $f\in C(Z)$.
\end{Pro}

\begin{proof} Without loss of generality, we can assume that $f_0(x_0)=1$ for some $x_0\in Z$. Take any $F$=$\{x_1,x_2,...,x_n\}$ $\ci Z$. Assume that $|f_0(x_1)|\leq |f_0(x_2)|\leq ... \leq |f_0(x_n)|=a\neq 0$ (otherwise, add more points in $F$ so that $a\neq 0$). Choose now $x_{n+1}$ such that $1<|f_0(x_{n+1})|$ and $|f_0(x_{n+1})|>|f_0(x_n)|$ (this is possible as $f_0$ is not attaining its supremum). Let $F'=\{x_0,x_1,x_2,...,x_n,x_{n+1}\}$ and $|f_0(x_{n+1})|=b$. We show that the pairs $(\delta_{x_0},\rho_{F})$ and $(\delta_{x_0},\rho_{F'})$ on $Y$ have two distinct extensions $\phi$ and $\psi$, respectively, where  $\phi= sgn\big(f_0(x_n)\big)\frac{\delta_{x_n}}{a}$  and $\psi= sgn\big(f_0(x_{n+1})\big)\frac{\delta_{x_{n+1}}}{b}$. Clearly, $\phi|_Y=\psi|_Y=\delta_{x_0}$. Note that $|\delta_{x_0}(\alpha f_0)|=|\alpha f_0(x_0)|=|\alpha|\leq \frac{1}{a}$ whenever $\rho_{F}(\alpha f_0)\leq 1$. Since $\rho_F\left( \frac{f_0}{a}\right)=1$ and $|\frac{f_0}{a}(x_0)|=\frac{1}{a}$, we have $\chi_{\rho_{F}}^Y(\delta_{x_0})=\frac{1}{a}$. It is easy to prove that $\chi_{\rho_{F}}^X(\phi)=\frac{1}{a}$. Similarly, we can show that $\chi_{\rho_{F'}}^Y(\delta_{x_0})=\chi_{\rho_{F'}}^X(\psi)=\frac{1}{b}$. Which completes the proof.
\end{proof}

\begin{thm} 
Let $Y$= $span(f_0)$ be a subspace of $X=C_p(Z)$. Then $Y$ has the property-SNP with respect to $\mc{P}=\{\rho_F:F\ci Z \text{ is finite}\}$ if and only if $f_0$ attains its supremum at exactly one point, where $\rho_F(f)=\max_{x\in F}|f(x)|$ for all $f\in C(Z)$.
\end{thm}

\begin{proof} 
Suppose $f_0$ attains its supremum at exactly one point. Assume that $x_0\in Z$ such that $\sup_{x\in Z}|f_0(x)|=f_0(x_0)=1$. Since the dimension of $Y$ is one, we have $Y^*$= span$(\delta_{x_0})$, where $\delta_{x_0}(f)=f(x_0)$. It is now enough to find a finite set $F_0\ci Z$ such that every pair $(\delta_{x_0}, \rho_{F})$ has a unique extension whenever $F_0\ci F$. Let $F_0=\{x_0\}$ and let $F_0\ci F$ be finite. Then $\chi_{\rho_{F}}^X(\delta_{x_0})=1$. Suppose $\widetilde{\phi}=\sum_{i=1}^{m}\alpha_i\delta_{x_i}$ is any extension of the pair $(\delta_{x_0}, \rho_{F'})$. Observe that all $x_i\in F$ as $\chi_{\rho_F}^X(\widetilde{\phi})<\iy$. Clearly, $|\widetilde{\phi}(f)|\leq \sum_{i=1}^m|\alpha_i|$ whenever $\rho_F(f)\leq 1$. Consider now an $h\in C(Z)$ such that $\|h\|_\iy\leq 1$ and $$h(x_i)=\Big\{\begin{array}{ll} 1 & \text{ if }\alpha_i>0 \\ -1 &\text{ if }\alpha_i<0.\end{array}$$ Then $\widetilde{\phi}(h)=\sum_{i=1}^m|\alpha_i|$. Since $\chi_{\rho_F}^X(\widetilde{\phi})=\chi_{\rho_{F}}^X(\delta_{x_0})=1$, we have $\sum_{i=1}^m|\alpha_i|=1$.

{\sc Step 1:} We show that $x_0=x_i$ for some $i$. Suppose $x_i\neq x_0$ for all $i$. Then
$$\begin{aligned}
1=|f_0(x_0)|&=|\delta_{x_0}(f_0)|=|\widetilde{\phi}(f_0)|=\left|\sum_{i=1}^m\alpha_i f_0(x_i)\right|\\
&\leq \left(\max_{1\leq i\leq m}|f_0(x_i)|\right)\sum_{i=1}^m|\alpha_i|\\
&<1.1=1\text{  [since }f_0(x_i)<1\text{ for all }x_i\neq x_0].
\end{aligned}$$

{\sc Step 2:} Assume that $x_1=x_0$. Next, we show that $f_0(x_i)=0$ for all $i>1$. Note that,
\beqa
|\alpha_1|+\sum_{i=2}^m|\alpha_i|=1
&=&\left|\sum_{i=1}^m \alpha_i f_0(x_i)\right|\\
&\leq& \sum_{i=1}^m |\alpha_i||f_0(x_i)|\\
&=&|\alpha_1||f_0(x_1)|+\sum_{i=2}^m |\alpha_i||f_0(x_i)|\\
&<&|\alpha_1|+\sum_{i=2}^m|\alpha_i|\text{ [since }|f_0(x_i)|<1 \text{ for all }i\neq 1]
\eeqa

Hence, $f_0(x_i)=0$ for all $i\geq 2$. This implies $\widetilde{\phi}(f_0)=f_0(x_0)=\delta_{x_0}(f)$. Hence, $Y$ has the property-SNP.	

Conversely, if $Y$ has the property-SNP, then by Proposition \ref{prop 1}, $f_0$ must achieve its supremum. Suppose that there exist two points $x_1$ and $x_2$ such that $\|f\|_\iy=|f(x_1)|=|f(x_2)|=1$. Then the pair $(\delta_{x_1},\rho_F)$ has two distinct extensions $\delta_{x_1}$ and $\delta_{x_2}$ (or $-\delta_{x_2}$ if $f(x_2)=-f(x_1)$) for all $F\supseteq \{x_1,x_2\}$. Hence, the result holds.  
\end{proof}
 
If the dimension of $Y$ is two, then we only have a necessary condition.

\begin{Pro}\label{P5} 
Let $Y$ be a two-dimensional subspace of $X=C_p(Z)$ which has the property-SNP. Then any non-zero element of $Y$ cannot attain its supremum on more than two points. 
\end{Pro}

\begin{proof} Let $Y$= span$\{f_1,f_2\}$, where $f_1,f_2\in C_p(Z)$. Also, assume that there exists $x_1,x_2\text{ and }x_3\in Z$ with $|f_1(x_1)|=|f_1(x_2)|=|f_1(x_3)|=\sup\limits_{x\in Z}|f_1(x)|=1$. 

{\sc Case 1:~} Suppose that $|f_2(x_1)|=|f_2(x_2)|$. 

{\sc Case 1.1:~} When $f_1(x_1)=f_1(x_2)$ and $f_2(x_1)=f_2(x_2)$. 

Then for any $F\supseteq\{x_1,x_2\}$, we have $\delta_{x_1}$ and $\delta_{x_2}$ two distinct extensions of $(\delta_{x_1},\rho_F)$. Hence, this case cannot occur.

{\sc Case 1.2:~} When $f_1(x_1)=-f_1(x_2)$ and $f_2(x_1)=-f_2(x_2)$. 

Then for any $F\supseteq\{x_1,x_2\}$, we have $\delta_{x_1}$ and $-\delta_{x_2}$ two distinct extensions of $(\delta_{x_1},\rho_F)$. Hence, this case cannot occur.

{\sc Case 1.3:~} When $f_1(x_1)=-f_1(x_2)$ and $f_2(x_1)=f_2(x_2)$. 

This condition ensures that, $\delta_{x_1}$ and $\delta_{x_2}$ are linearly independent. 

{\sc Case 1.4:~} When $f_1(x_1)=f_1(x_2)$ and $f_2(x_1)=-f_2(x_2)$.

This condition ensures that, $\delta_{x_1}$ and $\delta_{x_2}$ are linearly independent.

{\sc Case 2:~} Suppose that $|f_2(x_1)|\neq |f_2(x_2)|$. 

This condition ensures that, $\de_{x_1}$ and $\de_{x_2}$ are linearly independent. 

It now remains to prove that if $\delta_{x_1}$ and $\delta_{x_2}$ are linearly independent, then $Y$ does not have the property-SNP. If this holds, then $Y^*$= span$(\delta_{x_1},\delta_{x_2})$. So, there exists $a,b\neq 0$ such that $\delta_{x_3}=a\delta_{x_1}+b\delta_{x_2}$. Now, for any $F\supseteq \{x_1,x_2,x_3\}$, $\delta_{x_3}$ and $a\delta_{x_1}+b\delta_{x_2}$ are two distinct extensions of $(\delta_{x_3},\rho_F)$. Hence, the result follows.\end{proof}

We do not know whether an analogous result to Proposition~\ref{P5} is true if $Y$ is a finite co-dimensional subspace of $X$. In \cite[Theorem~3.2]{RP}, Phelps derived the similar result in Banach spaces, we need suitable techniques to implement these ideas in locally convex spaces.

\section*{Acknowledgement}
The authors thank Prof. T.S.S.R.K. Rao for drawing our attention towards the articles [R. P. Agnew, A. P. Morse, {\it Extensions of linear functionals, with applications to limits, integrals, measures, and densities}, Ann. of Math.(2), 39(1) (1938), 20--30] and [P. Bandyopadhyay, A. K. Roy, {\it Nested sequences of balls, uniqueness of Hahn-Banach extensions and the Vlasov property}, Rocky Mountain J. Math., 33(1) (2003), 27--67], which is in fact a starting point of this investigation. 

The authors declare that there is no conflict of interest.

\Addresses
\end{document}